\documentclass[11pt]{amsart}
\usepackage{amsmath}
\usepackage{amssymb,amscd}
\usepackage{graphicx}
\usepackage{tikz}
\usepackage{tikz-cd}
\usepackage{mathdots}
\usepackage{resizegather}
\usepackage{mathrsfs}      
\usepackage{graphicx}
\usepackage{mathabx}
\usepackage{adjustbox}
\usepackage{microtype} 

\usepackage{colonequals}

\theoremstyle{plain}
\newtheorem{theorem}{Theorem}[section]
\newtheorem{lemma}[theorem]{Lemma}
\newtheorem{proposition}[theorem]{Proposition}
\newtheorem{corollary}[theorem]{Corollary}

\newtheorem{innercustomgeneric}{\customgenericname}
\providecommand{\customgenericname}{}
\newcommand{\newcustomtheorem}[2]{%
	\newenvironment{#1}[1]
	{%
		\renewcommand\customgenericname{#2}%
		\renewcommand\theinnercustomgeneric{##1}%
		\innercustomgeneric
	}
	{\endinnercustomgeneric}
}

\newcustomtheorem{customthm}{Theorem}

\theoremstyle{definition}
\newtheorem{definition}[theorem]{Definition}
\newtheorem{example}[theorem]{Example}
\newtheorem{remark}[theorem]{Remark}

\numberwithin{equation}{section}

\def \begineq{\begin{equation}}
\def \endeq{\end{equation}}

\def \bb{\mathbb}

\def \CC{{\bb{C}}}

\def \QQ{{\bb{Q}}}
\def \RR{{\bb{R}}}
\def \ZZ{{\bb{Z}}}

\def \Aut{{\rm Aut}}

\def \qed{\hfill $\square$ \vspace{0.03in}}

\begin{document}

\title[ Resolution of singularities and equivariant cobordism ]{ Resolution of singularities of toric orbifolds and equivariant cobordism of contact toric manifolds}

\author[K. Brahma]{Koushik Brahma}
\address{Department of Mathematics, Indian Institute of Technology Madras, Chennai 600036, India}
\email{koushikbrahma95@gmail.com}

\author[S. Sarkar]{Soumen Sarkar}
\address{Department of Mathematics, Indian Institute of Technology Madras, Chennai 600036, India}
\email{soumen@iitm.ac.in}

\author[S. Sau]{Subhankar Sau}
\address{Department of Mathematics, Indian Institute of Technology Madras, Chennai 600036, India}
\email{subhankarsau18@gmail.com}

\date{\today}



\abstract 
Toric orbifolds are a generalization of simplicial projective toric varieties. In this paper, we show that there is a resolution of singularities of a toric orbifold. In a different category, the class of quasi-contact toric manifolds contains the class of good contact toric manifolds. We prove that a quasi-contact toric manifold is equivariantly a boundary. Moreover, we conclude that good contact toric manifolds and generalized lens spaces are equivariantly boundaries. 
\endabstract

\maketitle

\section{Introduction}\label{intro}
Toric manifolds were introduced in the pioneering paper \cite{DJ} by Davis and Januszkiewicz. These manifolds are topological generalization of smooth projective toric varieties. The paper \cite{DJ} considered the standard action of the compact abelian torus $T^n$ on $\CC^n$ as the local model to define toric manifolds. Briefly, a $2n$-dimensional smooth manifold with a locally standard $T^n$-action is called a toric manifold if the orbit space has the structure of an $n$-dimensional simple polytope.
Moreover, \cite[Section 7]{DJ} initiated the notion of toric orbifolds generalizing toric manifolds. These orbifolds are explicitly studied in \cite{PS} with the name `quasitoric orbifolds'. Examples of toric orbifolds include simplicial projective toric varieties.

On the other hand, contact toric manifolds are odd-dimensional analogues of symplectic toric manifolds with Hamiltonian torus actions, see \cite{BoGa}. Lerman provided a complete classification of compact connected contact toric manifolds in \cite[Theorem 2.18]{Lerman}. Motivated by the work of Davis and Januszkiewicz \cite{DJ}, Luo  discussed the combinatorial construction of good contact toric manifolds and studied some of their topological properties in \cite{Luo-Thesis}.

Recently,  the construction of toric manifolds in \cite{DJ} has been extended to introduce the notion of \emph{locally k-standard $T$-manifolds}, in \cite{SaSo} where $T$ is a compact abelian torus. The paper \cite{SaSo} considers the invariant subset $\CC^{n}\times (S^1)^k$ of $\CC^{n+k}$ with respect to the standard $T^{n+k}$-action, and calls the restricted $T^{n+k}$-action on  $\CC^{n}\times (S^1)^k$ as the local model to define a locally $k$-standard $T$-manifold. In particular, when $k=1$ we call this manifold a \emph{quasi-contact toric manifold}. Briefly, a $(2n+1)$-dimensional smooth compact $T^{n+1}$-manifold $N$ is called a quasi-contact toric manifold if each point of $N$ has an invariant neighbourhood which is diffeomorphic to an invariant open subset of  $\CC^{n}\times S^1$ such that the orbit space ${N}/{T^{n+1}}$ has the structure of an $n$-dimensional simple polytope. The article \cite{SaSo} shows that the category of quasi-contact toric manifolds contains all good contact toric manifolds. We use the term `quasi-contact toric manifold' as the topological counterpart of `good contact toric manifold'.

Resolution of singularity is a widely used tool in algebraic geometry nowadays, see \cite{Kol, cut}. One can trace back to Newton and Riemann for the idea of resolution of singularities of the curves. If $X$ is a singular toric variety then there is a resolution of singularities of $X$, see \cite[Chapter 2]{Ful}. In algebraic topology, \cite{GP} discusses the resolution of singularities of four dimensional toric orbifolds. We note that there are infinitely many toric orbifolds which are not toric varieties. For example, an equivariant connected sum of two weighted projective spaces is a toric orbifold but not a toric variety. In this article, we discuss the resolution of singularities for any toric orbifolds. Our technique is different than the proof of the resolution of singularities of a singular toric variety.

Lev Pontryagin introduced the notion of cobordism in \cite{Pon47}. Cobordism theory discusses about an equivalence on the same dimensional compact manifolds using the concept of the boundary of a manifold. Considering the disjoint union as addition, one can get an abelian group structure on these equivalence classes. If a manifold is a boundary of a manifold with boundary then it is called null-cobordant. The cobordism groups could be computed through homotopy theory using the Thom construction, see \cite{Tho}. We have a complete information about oriented, non-oriented and complex cobordism rings.
The notion of equivariant cobordism was introduced in early 1970's, see \cite{ Hook73-1,Hook73, Peter}.
The equivariant cobordism rings for some of the nontrivial groups can be found in \cite{Han}. However, they are not known for any nontrivial groups. One of the key reason is that in equivariant category, the Thom transversality theorem does not hold. Thus the equivariant cobordism theory cannot be reduced to the equivariant homotopy theory. 
In this article, we study the equivariant cobordism of quasi-contact toric manifolds, good contact toric manifolds and generalized lens spaces.

The paper is organized as follows.
 In Section \ref{sec_res_od_sing}, we recall the notion of a toric manifold and a toric orbifold following \cite{DJ}.  
 Then we recall the notion of blowup of a simple polytope as well as blowup of a toric orbifold following \cite{BSS}. We discuss the orbifold singularity on any face in the orbit space of a toric orbifold. Then we apply the techniques of blowup of a toric orbifold for resolution of singularity and prove the following result. 
\begin{customthm}{A}[Theorem \ref{thm_res_sing}]
	For a toric orbifold $M(P,\lambda)$, there exists a resolution of singularities $$M(P^{(d)},\lambda^{(d)})\to \dots \to M(P^{(1)},\lambda^{(1)})\to M(P,\lambda)$$ such that $M(P^{(d)},\lambda^{(d)})$ is a toric manifold, the toric orbifold $M(P^{(i)},\lambda^{(i)})$ is obtained by a blowup of $M(P^{(i-1)},\lambda^{(i-1)})$ for $i=1,2,\dots,d$ and the arrows $\rightarrow$ indicate the associated blowups.
\end{customthm}
 
 In Section \ref{sec_tctm}, we recall the concept of quasi-contact toric manifolds and discuss some of their properties. The idea of the construction of these spaces is similar to the construction of toric manifolds introduced by Davis and Januszkiewicz \cite{DJ}. 
 Generalized lens spaces \cite{SS} and good contact toric manifolds are some well known examples of quasi-contact toric manifolds, see Example \ref{ex_gen_lens_sp} and Example \ref{ex_gd_ct_toric_mfd}.
 Then we prove the following.
 \begin{customthm}{B}[Theorem \ref{thm_zoro_cob}]
 Let $N$ be a $(2n+1)$-dimensional quasi-contact toric manifold. Then there is a smooth oriented $T^{n+1}$-manifold with boundary $M$ such that $\partial{M}$ is equivariantly diffeomorphic to $N$.
\end{customthm}
 As a consequence, one can obtain that good contact toric manifolds and generalized lens spaces are equivariantly cobordant to zero, and hence all the Stiefel-Whitney numbers of these manifolds are zero. We note that a lens space is a contact toric manifold.

\section{Resolution of singularities of toric orbifolds}\label{sec_res_od_sing}
In this section, we discuss the constructive definition of toric manifolds and toric orbifolds following \cite{DJ}. 
These spaces are even dimensional orbifolds and equipped with half-dimensional locally standard torus actions where the orbit spaces are simple polytopes. We also recall the notion of blowup of a toric orbifold along certain invariant suborbifolds. We discuss when an open convex bounded subset of $\RR^k$ contains an integral point. We prove that there is a resolution of singularities of a toric orbifold.

A convex hull of finitely many points in $\RR^n$ for some $n \in \ZZ_{\geq 1}$ is called a convex polytope. If each vertex (zero dimensional face) in an $n$-dimensional convex polytope is the intersection of $n$ facets (codimension one faces), then the polytope is called a simple polytope. One can find the basic properties of simple polytopes in \cite{Zi, BP-book}. For a simple polytope $P$ we denote the vertex set of $P$ by $V(P):=\{b_1,\dots,b_m\}$ and the facet set by $\mathcal{F}(P):=\{F_1,\dots,F_r\}$ throughout this paper.

\begin{definition}
Let $P$ be an $n$-dimensional simple polytope and $\lambda \colon \mathcal{F}(P) \rightarrow \mathbb{Z}^{n}$ a function such that $\lambda(F_i)$ is primitive for $i \in \{1, \ldots, r\}$ and
\begin{equation}\label{Eq_lin ind vec}
\{\lambda(F_{i_1}),\dots,\lambda(F_{i_k})\}~ \text{is linearly independent if } \bigcap_{j=1}^k F_{i_j}\neq \varnothing.
 \end{equation}
\noindent Then $\lambda$ is called an $\mathcal{R}$-{characteristic function} on $P$, and the pair $(P,\lambda)$ is called an $\mathcal{R}$-{characteristic pair}.

If the set $\{\lambda(F_{i_1}),\dots,\lambda(F_{i_k})\}$ in \eqref{Eq_lin ind vec} spans a $k$-dimensional unimodular submodule of $\ZZ^{n}$, then $\lambda$ is called a characteristic function. In this case, the pair $(P,\lambda)$ is called a characteristic pair.
\end{definition}

Observe that a characteristic function is an $\mathcal{R}$-characteristic function. We denote $\lambda(F_i)$ by $\lambda_i$ and call it the $\mathcal{R}$-{characteristic vector} or characteristic vector assigned to the facet $F_i$ according to the situation.

Next, we recall the construction of a toric orbifold from an $\mathcal{R}$-characteristic pair $(P,\lambda)$. We denote the standard $n$-dimensional torus by $T^n$. Note that $\mathbb{Z}^n$ is the standard $n$-dimensional lattice in the Lie algebra of $T^n$. Also, each $\lambda_i \in \ZZ^n$ determines a line in $\mathbb{R}^n(=\ZZ^n \otimes_{\ZZ} \RR)$, whose image under the exponential map $exp \colon \mathbb{R}^n \to T^n$ is a circle subgroup, denoted by $T_i$.

Let $F$ be a codimension-$k$ face of $P$ where $0 < k \leq n$ and $\mathring{F}$ the relative interior of $F$. Then $F=\bigcap_{j=1}^k F_{i_j}$ for some unique facets $F_{i_1}, \dots, F_{i_{k}}$ of $P$.
  Let 
 \begin{equation*}
 T_F := \big< T_{i_1}, \dots, T_{i_{k}} \big>    
 \end{equation*} 
 be the $k$-dimensional subtorus of $T^n$ generated by $ T_{i_1}, \dots, T_{i_{k}}$. We define $T_P: =1 \in T^n$.

We consider the identification space $M(P,\lambda): =(T^n \times P) / \sim$, where the equivalence relation $\sim$ is defined by
\begin{equation*}
(t,x) \sim (s,y)~ \text{if and only if}~ x=y\in \mathring{F} ~\text{and}~ t^{-1}s \in T_F.
\end{equation*} 
The identification space $M(P,\lambda)$ has a $2n$-dimensional orbifold structure with a natural $T^n$-action induced by the group operation on the first factor of $T^n\times P$. The projection onto the second factor gives the orbit map
\begin{equation}\label{orbit map}
\pi \colon M(P,\lambda)  \to P \text{ defined by } [t,x]_{\sim} \mapsto x \nonumber
\end{equation}
where $[t,x]_{\sim}$ is the equivalence class of $(t,x)$.
An explicit construction of the orbifold structure on $M(P,\lambda)$ is discussed in \cite{PS} and \cite{GP}. Poddar and the second author gave the axiomatic definition of (quasi)toric orbifolds and showed that the constructive and axiomatic definitions of toric orbifolds are equivalent, see \cite[Proposition 2.9]{PS}.

We note that if $\lambda$ is a characteristic function, then $M(P, \lambda)$ is a toric manifold, see \cite{DJ}. The equivalence of the axiomatic and constructive definitions of toric manifolds is discussed in \cite{DJ}.

Now, we review the notion of blowup of a simple polytope $P$ as well as blowup of the toric orbifold $M(P, \lambda)$ along an invariant suborbifold. 
\begin{definition}\label{def_qtoric_orb}
	Let $P$ be an $n$-dimensional simple polytope in $\RR^n$ and $F$ a face of $P$. Take an $(n-1)$-dimensional hyperplane $H$ in $\RR^n$ such that $V(F)$ is a subset of the open half space $H_{>0}$  and $V(P) \setminus V(F)$ is a subset of the other open half space $H_{<0}$. Then $\widebar{P}:= P \cap H_{\leq 0}$ is called a blowup of $P$ along $F$.
\end{definition}

Note that $\widebar{P}\subset P$ and $\widebar{P}$ is an $n$-dimensional simple polytope. If $\mathcal{F}(P)=\{F_1,\dots,F_r\}$  and dim$(F)<(n-1)$ then the facets of $\widebar{P}$ is $\mathcal{F}(\widebar{P}):=\{\widebar{F}_1, \dots, \widebar{F}_r, \widebar{F}_{r+1}\}$ where
\begin{align}\label{eq_facet}
\widebar{F}_i := \begin{cases} F_i \cap \widebar{P} \quad &\text{for }	i=1, \dots, r,\\ H \cap P \quad &\text{for }	i=r+1. \end{cases}
\end{align}
\noindent In this paper, a blowup $\widebar{P}$ of $P$ might also be denoted by $P^{(1)}$ and a blowup of ${P}^{(\ell)}$ is denoted by $P^{(\ell+1)}$, for $\ell \geq 1$.

\begin{definition}\label{ex_blow_up}
Let $\widebar{P}$ be a blowup of a simple polytope $P$ along a face $F$. Let $(P, \lambda)$ and $(\widebar{P}, \bar{\lambda})$ be two $\mathcal{R}$-characteristic pairs such that $$\bar{\lambda}(\widebar{F}_i)=\lambda(F_i) \quad \text{ if }\widebar{F}_i=F_i \cap \widebar{P} \text{ for } F_i\in \mathcal{F}(P).$$ Then the toric orbifold $M(\widebar{P}, \bar{\lambda})$ is called a blowup of the toric orbifold $M(P, \lambda)$ along the suborbifold $\pi^{-1}(F)$.
\end{definition}

\begin{example}\cite[Example 5.2]{BSS}\label{rmk_char_fnc}
  Let $(P, \lambda)$ be an $\mathcal{R}$-characteristic pair and $\widebar{P}$ a blowup of $P$ along a codimension $k$ face $F$ where $0 < k \leq n$. Then $F=\bigcap_{j=1}^{k} F_{i_j}$ for some unique facets $F_{i_1}, \dots, F_{i_{k}}$ of $P$. 
Let $$\QQ(F) := \{ (c_1, \ldots, c_{k}) \in \QQ^{k}  ~|~ c_j \in \QQ \setminus \{0\} ~\mbox{and}~ \textbf{0} \neq \sum_{j=1}^{k} c_j \lambda_{i_j} \in \ZZ^n\}.$$ 
We define $\bar{\lambda} \colon \mathcal{F}(\widebar{P}) \to \mathbb{Z}^n$ by
\begin{align}\label{characteristic vector after blowup}
\bar{\lambda}(\widebar{F_i}):= 
     \begin{cases}
       \lambda_i &\quad\text{if}~ \widebar{F}_i = F_i \cap \widebar{P} \text { for } i=1,2,\dots,r\\
       \text{prim}(\sum_{j=1}^{k} c_j \lambda_{i_j}) &\quad\text{if}~ \widebar{F}_i=\widebar{F}_{r+1}~ \text{and}~(c_1, \dots, c_{k}) \in \QQ(F)
     \end{cases}
\end{align}
where prim$(\alpha)$ indicates the primitive vector of $\alpha \in \mathbb{Z}^n \setminus \{\textbf{0}\}$.
Then $\bar{\lambda}$ is an $\mathcal{R}$-{characteristic function} on $\widebar{P}$. The pair $(P,\lambda)$ and $(\widebar{P},\widebar{\lambda})$ satisfy Definition \ref{ex_blow_up}. Thus $M(\widebar{P},\bar{\lambda})$ is a blowup of the toric orbifold $M(P,\lambda)$ along the suborbifold $\pi^{-1}(F)$. 
\end{example}

Let  $(P, \lambda)$ be an $\mathcal{R}$-characteristic pair and $F$ a face as in Example \ref{rmk_char_fnc}.
Let $\lambda_{i_j} = \lambda(F_{i_j})$ for all $j=1,2,\dots,k$. Then the group
$$G_F(P,\lambda):= \frac{(\big< \lambda_{i_1},\lambda_{i_2},\dots,\lambda_{i_k} \big> \bigotimes_{\ZZ}\RR)\cap \ZZ^{n}}{\big< \lambda_{i_1},\lambda_{i_2},\dots,\lambda_{i_k} \big>}$$
is finite and abelian, where ${\big< \lambda_{i_1},\lambda_{i_2},\dots,\lambda_{i_k} \big>}$ is the submodule generated by $ \{\lambda_{i_1}, \lambda_{i_2}, \dots, \lambda_{i_k}\}$. This group measures the order of singularity of points in $\pi^{-1}(\mathring{F}) \subseteq M(P, \lambda)$. For simplicity, we may denote $G_F(P,\lambda)$ by $G_F$, whenever the context is clear. Note that if $F$ is a face of $F'$ then $|G_{F'}|$ divides $|G_F|$ and the group $G_F$ is trivial if $M(P,\lambda)$ is a toric manifold. 

We recall the definition of the volume of a parallelepiped in an inner product space. Let $\{u_1,u_2,\dots,u_k\}$ be an orthonormal basis in the $k$-dimensional real inner product space $V$. Let 
$$C_V:=\{\sum_{i=1}^{k}r_iu_i \in V ~|~ 0\leq r_i \leq \alpha_i, r_i\in \RR ~\mbox{and} ~ 0 < \alpha_i\in \RR\}.$$ Then $C_V$ is a $k$-dimensional parallelepiped and $$\mbox{vol}(C_V) =\alpha_1\alpha_2\cdots \alpha_k.$$

For a face $F=\bigcap_{j=1}^{k} F_{i_j}$, we consider the vector space $V_F=(\big< \lambda_{i_1},\lambda_{i_2},\dots,\lambda_{i_k} \big> \bigotimes_{\ZZ}\RR)$.
Let $\{v_1,\dots,v_k\}$ be a basis of the lattice $$L_F := (\big< \lambda_{i_1},\lambda_{i_2},\dots,\lambda_{i_k} \big> \bigotimes_{\ZZ}\RR)\cap \ZZ^{n}$$ in $V_F$. So $L_F=\ZZ v_1+\ZZ v_2+\dots+\ZZ v_k$.
Define $$C:=\{\sum_{i=1}^{k}r_iv_i~~|~~ 0\leq r_i < 1, r_i\in \RR\}.$$ This $C$ is called a fundamental parallelepiped for the lattice $L_F$. It is well known that $\mbox{vol}(C)$ is independent of the bases of $L_F$.
Let  
\begin{equation}\label{eq_para}
    C_F=\{\sum_{j=1}^{k}r_j\lambda_{i_j} ~~|~~ 0\leq r_j < 1, r_j\in \RR\}.
\end{equation}

Then $|G_F|$ measures the volume of the $k$-dimensional parallelepiped $C_F\subset V_F$ made by the $k$ vectors $\{\lambda_{i_1}, \lambda_{i_2}, \dots, \lambda_{i_k}\}$, where $|G_F|$ is the order of the group $G_F$.
Mathematically, one can write $$\mbox{vol}(C_F)=|G_F|\times \mbox{vol}(C).$$
The following Minkowski theorem says when an open convex subset of $\RR^n$ contains a non-zero lattice point.

\begin{proposition}\cite[Theorem 5.2]{OLD}\label{thm_min}
    Let $X$ be an open convex subset of a $k$-dimensional inner product space $V\subset \RR^n$ and $C$ the fundamental parallelepiped for the lattice $V\cap\ZZ^n$. If $X$ is symmetric about origin and has volume greater than $2^k vol(C)$, then $X$ contains a point ${\mathbf b}=(b_1,b_2,\dots,b_n)\in V\cap \ZZ^n$ other than the origin.
\end{proposition}

\begin{lemma}\label{lem_non_zero_int_pt}
Let $\lambda$ be an $\mathcal{R}$-characteristic function on $n$-dimensional simple polytope $P$ and $F=\bigcap_{j=1}^k F_{i_j}$ a face in $P$. If $|G_F|>1$ then there exists a non-zero ${\lambda}_F\in \ZZ^{n}$ such that ${\lambda}_F=\sum_{j=1}^{k}c_j\lambda_{i_j}$ with $ |c_j| < 1$ and $c_j\in \QQ$. Moreover, if $|G_{F'}|=1$ for every face $F'$ properly containing $F$ in $P$, then there exists a non-zero ${\lambda}_F\in \ZZ^{n}$ such that ${\lambda}_F=\sum_{j=1}^{k}c_j\lambda_{i_j}$ where $0 < |c_j| < 1$ and $c_j\in \QQ$ for $j=1,2,\dots,k$. 
\end{lemma}

\begin{proof}
The parallelepiped $C_F$ constructed in \eqref{eq_para} is convex but not symmetric about the origin. Consider the union of $2^k$ many parallelepiped defined by $$\widetilde{C}_F=\{\sum_{j=1}^{k}r_j\lambda_{i_j} ~~|~~ -1 < r_j < 1, r_j\in \RR\}.$$
This is an open parallelepiped which is convex and symmetric about the origin such that $$\mbox{vol}(\widetilde{C}_F)=2^k\mbox{vol}(C_F)=2^k|G_F|\mbox{vol}(C)>2^k\mbox{vol}(C).$$ Then by Proposition \ref{thm_min}, there exists a non-zero point ${\lambda}_F\in \ZZ^n\cap \widetilde{C}_F$. 
Then there exist $c_j\in \QQ$ with $|c_j| < 1$ for all $j\in\{1,2,\dots,k\}$ such that ${\lambda}_F=\sum_{j=1}^{k}c_j\lambda_{i_j}$. 

Now for the second part, without loss of generality, we may assume that $c_1=0$. Consider the face $F':=\cap_{j=2}^{k}F_{i_j}$. Then $F'$ contains the face $F$ properly. By our assumption, we have $|G_{F'}|=1$. Then $\{\lambda_{i_2},\lambda_{i_3},\dots,\lambda_{i_k}\}$ is a $\ZZ$-basis of the lattice
 $(\big<\lambda_{i_2},\lambda_{i_3},\dots,\lambda_{i_k}\big>\bigotimes_{\ZZ}\RR)\cap \ZZ^{n}$. Observe that this contradicts the existence of a non-zero lattice point ${\lambda}_F=\sum_{j=2}^{k}c_j\lambda_{i_j}$ such that $|c_j|<1$ for $j\in\{2,3,\dots,k\}$. Therefore $0<|c_1|<1$. Similarly, we can show $0<|c_j|<1$ for $j\in\{2,3,\dots,k\}$.
\end{proof}

\begin{remark}
Let $F=\bigcap_{j=1}^k F_{i_j}$ be a face of $P$ and $\lambda$ an $\mathcal{R}$-characteristic function on $P$.
\begin{enumerate}
	\item If $|G_F|=1$, then 
	$$(\big<\lambda_{i_1},\lambda_{i_2},\dots,\lambda_{i_k}\big>\bigotimes_{\ZZ}\RR)\cap \ZZ^{n}=\ZZ \lambda_{i_1}+\ZZ \lambda_{i_2}+\dots+\ZZ \lambda_{i_k},$$
	and there does not exist a point of $(\big<\lambda_{i_1},\lambda_{i_2},\dots,\lambda_{i_k}\big>\bigotimes_{\ZZ}\RR)\cap \ZZ^{n}$ in the parallelepiped $C_F$ except the origin.
	\item  If $|G_F|>1$ then the non-zero lattice points of $(\big<\lambda_{i_1},\lambda_{i_2},\dots,\lambda_{i_k}\big>\bigotimes_{\ZZ}\RR)\cap \ZZ^{n}$ in the parallelepiped $C_F$ have a one-one correspondence to non-identity elements of the group $G_F$. So the total number of lattice points in the parallelepiped $C_F$ is $|G_F|$. Thus $\lambda_F$ obtained in the Lemma \ref{lem_non_zero_int_pt} is not unique if $|G_F|>2$.
\end{enumerate}
\end{remark}

\begin{theorem}\label{thm_res_sing}
	For a toric orbifold $M(P,\lambda)$, there exists a resolution of singularities
	$$M(P^{(d)},\lambda^{(d)})\to \dots \to M(P^{(1)},\lambda^{(1)})\to M(P,\lambda)$$ such that $M(P^{(d)},\lambda^{(d)})$ is a toric manifold, the toric orbifold $M(P^{(i)},\lambda^{(i)})$ is a blowup of $M(P^{(i-1)},\lambda^{(i-1)})$ for $i=1,2,\dots,d$ and the arrows $\to$ indicate the associated blowups.
\end{theorem}

\begin{proof}
	Let $$\mathcal{L}:=\{F ~\big{|}~ F \text{ is a face of } P \text{ and }|G_F(P,\lambda)|\neq 1\}.$$
	Define a partial order `$\leq$' on $\mathcal{L}$ 
by $F\leq F'$ if $F\subseteq F'$. 
	Without loss of generality, let $F$ be a maximal element in the set $\mathcal{L}$ with respect to the partial order `$\leq$'.
	
 Now consider the blowup of the polytope $P$ along the face $F$ and denote this blowup by $P^{(1)}$. If $\mathcal{F}(P)=\{F_1,\dots,F_r\}$, then we have $$\mathcal{F}(P^{(1)})=\{\widebar{F}_1, \dots, \widebar{F}_r, \widebar{F}_{r+1}\}$$ defined as in \eqref{eq_facet}. Let $F=\cap_{j=1}^k F_{i_j}$ where $F_{i_j}\in \mathcal{F}(P)$  and $\lambda_{i_j}=\lambda(F_{i_j})$ for $j\in \{1,2,\dots,k\}$. 
	We define a function ${\lambda}^{(1)} : \mathcal{F}(P^{(1)}) \to \ZZ^{n}$ by
\begin{align}\label{eq_lambda_blowup}
{\lambda}^{(1)}(\widebar{F_i}):= 
     \begin{cases}
       \lambda_i &\quad\text{if}~ \widebar{F}_i = F_i \cap \widebar{P} \text { for } i=1,2,\dots,r\\
       \text{prim}(\lambda_F)=\text{prim}(\sum_{j=1}^{k} c_j \lambda_{i_j}) &\quad\text{if}~ \widebar{F}_i=\widebar{F}_{r+1}
     \end{cases}
\end{align}
where ${\lambda}_F$ is obtained by Lemma \ref{lem_non_zero_int_pt}. Then $(P^{(1)},\lambda^{(1)})$ is an $\mathcal{R}$-characteristic pair.

There does not exist a face $F'$ containing $F$ such that $|G_{F'}|\neq 1$, by the definition of $\mathcal{L}$. Thus we can assume $0<|c_j|<1$ and $c_j\in \QQ$ for every $j\in\{1,2,\dots,k\}$ by Lemma \ref{lem_non_zero_int_pt}. 
Let  $$V(F):=\{b_1,b_2,\dots,b_{\ell}\} \text{ and } V(P)=\{b_1,b_2,\dots,b_{\ell},b_{\ell+1},\dots,b_m\}.$$ Note that there is a homeomorphism from $\widebar{F}_{r+1}$ to $F\times \Delta^{k-1}$ preserving the face structure. Let $V(\Delta^{k-1})=\{u_1,u_2,\dots,u_k\}$. We have $P^{(1)}\subset P$ from Definition \ref{def_qtoric_orb}. 
Then $$V(P^{(1)}):=\{(b_p,u_s) ~\big{|}~ 1\leq p\leq \ell; 1\leq s\leq k\}\cup\{b_{\ell+1},\dots,b_m\}.$$ 
We can assume $(b_p,u_s)$ is not a vertex of $\widebar{F}_{i_s}$ for $1\leq s\leq k$. For a vertex $b_p\in V(F) \subset P$, let $b_p= (\bigcap_{j=k+1}^{n}F_{i_j}) \bigcap F = \bigcap_{j=1}^{n}F_{i_j}$. Then for all $1\leq s\leq k$ and  $1\leq p\leq \ell$ $$(b_p,u_s)=\bigcap_{j=1~ j\neq s}^{n}\widebar{F}_{i_j}\bigcap\widebar{F}_{r+1}.$$

Now we calculate the singularity over each vertex in $P^{(1)}$. Note that
\begin{gather*}
{|G_{(b_p,u_s)}(P^{(1)},\lambda^{(1)})|=|\det[\lambda_{i_1},\dots,\widehat{\lambda_{i_s}},\dots,\lambda_{i_n},\text{prim}(\sum_{j=1}^{k} c_j \lambda_{i_j})]|=\frac{|c_j|}{d_F}|G_{b_p}(P,\lambda)|<|G_{b_p}(P,\lambda)| 
}
\end{gather*}
for  $1\leq p\leq \ell$ and  $1\leq s\leq k$ and $d_F$ is a positive integer satisfying $\lambda_F=d_F.\text{prim}(\lambda_F)$, and $$|G_{b_{\ell+i}}(P^{(1)},\lambda^{(1)})|=|G_{b_{\ell+i}}(P,\lambda)|  \text{ for } 1\leq i\leq m-\ell. $$ Therefore, in this above process if a vertex of $P$ is not in the face $F$ then the corresponding singularity remains same. Also corresponding to every vertex $b$ in $F$ we get exactly $k$ many new vertices in $\widebar{F}_{r+1}$ such that at
each of them the singularity is strictly less than the singularity on $b\in F\subset P$ in the given orbifold.

If $|G_F|=1$ for every faces of $P^{(1)}$ then $M(P^{(1)},\lambda^{(1)})$  is the desired resolution of singularities of $M(P,\lambda)$. Otherwise 
we repeat the above process on $M(P^{(1)},\lambda^{(1)})$ to obtain an $\mathcal{R}$-characteristic pair $(P^{(2)},\lambda^{(2)})$ where $P^{(2)}$ is a blowup of $P^{(1)}$, and $\lambda^{(2)}$ is defined similarly as in \eqref{eq_lambda_blowup} from $\lambda^{(1)}$. If $|G_F|=1$ for every faces of $P^{(2)}$ corresponding to the pair $(P^{(2)},\lambda^{(2)})$ then we are done. Otherwise, we repeat the process. Since the order of $G_F$ is finite for any face $F$ of $P$, this inductive process ends after a finite steps.
\end{proof}

\begin{figure}
    \begin{center}
        \begin{tikzpicture}[scale=0.8]
        \draw (0,0)--(1,-1)--(2,0)--(2,4)--(0,4)--(0,0);
        \draw (0,4)--(1,3)--(2,4);
        \draw[thick, blue] (1,3)--(1,-1);
        \draw [dashed] (0,0)--(2,0);
        \draw [<-] (1,.7)--(3,.7);
         \draw [<-] (1.1,3.5)--(3,3.5);
        \node at (3.8,3.5) {$(0,0,1)$};
        \node at (1,3) {$\bullet$};
        \draw[->] (1.5,-.2)--(1.5,-1.5); 
\node[below] at (1.5, -1.5) {\footnotesize$(0,0,1)$};
        \node at (3.2,.7) {$F$};
        \node at (1.3,2.9) {$b$};
        \draw [<-] (1.5,2)--(3,2);
        \node at (3.8,2) {$(1,2,0)$};
        \draw [<-] (.5,1.2)--(-1,1.2);
        \node at (-1.8,1.2) {$(1,0,0)$};
        
       \draw[thick, dotted] (.7,2.7)--(0,2.7); \draw[->](0,2.7)--(-1,2.7);
        \node at (-1.8,2.7) {$(0,1,0)$};
        
        \node at (1,-2.5) {(A)};

        \begin{scope}[xshift=250]
        \draw [->] (0,1.5)--(0,4);
        \draw [->] (1.5,0)--(4,0);
        \draw [->] (0,0)--(-1.5,-1.5);
        \node at (1.5,1.5) {$\bullet$};
        \draw[->,thick, red] (0,0)--(1.5,0);
        \draw[->,thick, green] (0,0)--(0,1.5);
        \draw [dashed] (1.5,0)--(1.5,1.5)--(0,1.5);
        \draw[->,thick, blue] (0,0)--(1.5,3);
        \draw[dashed] (1.5,3)--(3,3)--(1.5,0);
        \node[above] at (1.5,-.8) {$(1,0,0)$};
        \node[left] at (0,1.5) {$(0,1,0)$};
        \node[above] at (1.5,3) {$(1,2,0)$};
        \node[right] at (1.5,1.5) {$(1,1,0)$};
        \node at (2,-1.5) {(B)}; 
        \end{scope}

\begin{scope}[yshift=-200,xshift=100]
\draw[dashed, thick] (0,0)--(1,1)--(4,1); \draw[dashed, thick] (1,1)--(1,4);

\draw[thick] (0,0)--(3,0)--(3,3)--(0,3)--cycle;
\draw[thick] (3,0)--(4,1)--(4,4)--(3,3)--cycle;
\draw[thick] (4,4)--(1,4)--(0,3)--(3,3)--cycle;

\draw[->] (1.5,3.5)--(1.5,4.5); 
\node[above] at (1.5, 4.5) {\footnotesize$(0,0,1)$};

\draw[thick, dotted] (1.5,0.5)--(1.5,0); \draw[->] (1.5, 0)--(1.5, -0.5);
\node[below] at (1.5, -0.5) {\footnotesize$(0,0,1)$};

\draw[thick, dotted] (0.5,2)--(0,2); \draw[->] (0,2)--(-0.5, 2);
\node[left] at (-.5, 2) {\footnotesize$(1,0,0)$};

\draw[->] (3.5,2)--(4.5,2); 
\node[right] at (4.5, 2) {\footnotesize$(1,2,0)$};

\draw[->] (0.5,1.5)--(-.5,0.5); 
\node[left] at (-.5, 0.5) {\footnotesize$(1,1,0)$};

\draw[thick, dotted] (3.5,2.5)--(4,3); \draw[->] (4,3)--(4.5, 3.5);
\node[right] at (4.5, 3.5) {\footnotesize$(0,1,0)$};

\node at (2,-1.5) {(C)};
\end{scope}

        \end{tikzpicture}
    \end{center}
\caption{}
    \label{Fig_Example of GF}
\end{figure}

\begin{example}
Consider the $\mathcal{R}$-characteristic function and the edge $F$ of the triangular prism $P_1$ in Figure  \ref{Fig_Example of GF}(A). We have the $\ZZ$-module generated by $\{(1,2,0),(1,0,0)\}$. Then $$(\big<(1,2,0),(1,0,0)\big>\bigotimes_{\ZZ}\RR)\cap \ZZ^{3}=\{(x,y,0) ~~|~~ x,y\in \ZZ\} \cong \ZZ^2$$
which has a basis $\{(1,0,0),(0,1,0)\}$. In this case, $\mbox{vol}(C_F)=2$.
Thus for this edge $F$, we get $|G_F|=|\ZZ^2/\big<(1,2,0),(1,0,0)\big>| = 2$. Since the faces containing $F$ are facets of $P_1$, by Lemma \ref{lem_non_zero_int_pt} there exists a non-zero $\lambda_F \in \mathring{C_F} \cap \ZZ^3$. Here, the only non-zero lattice point in the parallelepiped $C_F$ is $\lambda_F = (1,1,0)$ which can be represented as  $$(1,1,0)=\frac{1}{2}(1,0,0)+\frac{1}{2}(1,2,0).$$

Also, similarly, one can calculate $|G_b|=2$ for $b \in V(F)$ and $|G_b|=1$ for $b \in V(P_1) - V(F)$. So, the maximal element in $\mathcal{L}$ for this case is $F$. Thus we blowup $P_1$ along the face $F$, and get the cube $P^{(1)}_1$ as in Figure \ref{Fig_Example of GF}(C). One can define $\lambda^{(1)}$ on $P^{(1)}_1$ as in \eqref{eq_lambda_blowup}. 
 Then in the toric orbifold $M(P^{(1)}_1,\lambda^{(1)})$, we have $|G_E|=1$ for every face $E$ of $P^{(1)}_1$. Thus $M(P^{(1)},\lambda^{(1)})$ is a toric manifold, and it is a resolution of singularities of $M(P_1,\lambda)$. \qed
\end{example}

\section{Equivariant cobordism of quasi-contact toric manifolds}\label{sec_tctm}
In this section, we recall the concept of quasi-contact toric manifolds and discuss some of their properties. Then we prove that any quasi-contact toric manifold is equivariantly the boundary of an oriented smooth manifold. In particular, good contact toric manifolds and generalized lens spaces are equivariantly cobordant to zero.

Consider the action  
$\alpha \colon T^{n+1}\times \CC^{n+1} \to \CC^{n+1}$ of $(n+1)$-dimensional torus $T^{n+1}$ on $\CC^{n+1}$ defined by
$$\alpha((t_1, \ldots, t_n, t_{n+1}), (z_1, \ldots, z_n, z_{n+1}))= 
(t_1z_1, \ldots, t_nz_n, t_{n+1}z_{n+1}).$$ 
Then $\CC^n \times S^1$
is a $T^{n+1}$-invariant subset of $\CC^{n+1}$, and the orbit space $(\CC^{n}\times S^1)/T^{n+1}$ is $(\RR_{\geq 0})^n$. The restriction $\alpha|_{T^{n+1}\times (\CC^{n}\times S^1)}$ is called the \emph{standard} $T^{n+1}$-action on $\CC^{n}\times S^1$.

\begin{definition}\label{def:axiom_top_cont}
A $(2n+1)$-dimensional smooth manifold $N$ with an effective $T^{n+1}$-action is called a quasi-contact toric manifold if 
the orbit space ${N}/{T^{n+1}}$ has the structure of an $n$-dimensional simple polytope and, for each point $p\in N$, there is
\begin{enumerate}
\item an automorphism $\theta\in \Aut(T^{n+1})$,
\item a $T^{n+1}$-invariant neighbourhood $U$ of $p$ in $N$ and a $T^{n+1}$-invariant open subset $U'$ of $\CC^n \times S^1$ such that there is a $\theta$-equivariantly diffeomorphism $f_{\theta}:U\to U'$ (that is, $f_{\theta}(tx)=\theta(t)f_{\theta}(x)$ for $(t,x)\in T^{n+1}\times U$).
\end{enumerate}
\end{definition}

 Note that $(2n+1)$-dimensional quasi-contact toric manifolds are locally $1$-standard $T$-manifolds of \cite{SaSo}.
Let 
$$\mathfrak{q} \colon N \to Q $$ 
be the orbit map where $Q$
is an $n$-dimensional simple polytope. Let $\mathcal{F}(Q)\colonequals \{E_1, \ldots, E_{\ell}\}$ be the set of facets of $Q$. Then each $N_{j} \colonequals \mathfrak{q}^{-1}(E_j)$ is a $(2n-1)$-dimensional $T^{n+1}$-invariant submanifold of $N$.
Then, the isotropy subgroup of $N_{j}$ is a circle subgroup $T_{j}$ of
$T^{n+1}$. The group $T_{j}$ is uniquely determined by a primitive vector
$\eta_j \in \ZZ^{n+1}$ for $j=1,2,\dots,\ell$; that is, we get a natural function
\begin{equation}\label{eq_axiomatic_lambda}
 \eta \colon \{E_1, \ldots, E_{\ell}\} \to \ZZ^{n+1}   
\end{equation}
defined by $\eta(E_j) = \eta_j$.

 We discuss the constructive definition of $(2n+1)$-dimensional quasi-contact toric manifolds on simple polytopes following \cite{SaSo}. Let $\mathcal{F}(Q)\colonequals \{E_1, \dots, E_\ell\}$ be the
set of facets of an $n$-dimensional simple polytope $Q$.

\begin{definition}\label{def_hyp_char_map}
A function $\xi \colon \mathcal{F}(Q) \to \ZZ^{n+1} $ is called a \emph{hyper characteristic function} if 
 $\big<\xi_{j_1}, \dots, \xi_{j_n}\big> \text{ is a rank $n$ unimodular submodule of } \ZZ^{n+1}$ whenever $E_{j_1}\cap\cdots\cap E_{j_n}\neq \emptyset$
 where $ \xi_j \colonequals \xi(E_j)$ for $j=1, \ldots, \ell$.
The pair $(Q, \xi)$ is called a \emph{hyper characteristic pair.}
\end{definition}

 Note that hyper characteristic function was defined on simplices in \cite[Definition 2.1]{SS} and Definition \ref{def_hyp_char_map} is the case $k=1$ in \cite[Definition 2.2]{SaSo}.
 
 Let $\xi$ be a hyper characteristic function on an $n$-dimensional simple polytope $Q$. For a point $p\in Q$, let $E_{j_1} \cap \cdots \cap E_{j_k}$ be the face of $Q$ containing $p$ in its relative interior. Then
 $\exp(\big< \xi_{j_1}, \dots,  \xi_{j_k}\big> \otimes_{\ZZ} \RR)$ is a $k$-dimensional subtorus of  $T^{n+1}$. We denote this subgroup by $T_p$. If $p$ belongs to the relative interior of $Q$, we denote $T_p=1$, the identity in $T^{n+1}$. We define the following identification space.
\begin{equation}\label{eq_constr_DJ}
N(Q, \xi)\colonequals  (T^{n+1}\times Q)/{\sim'}
\end{equation}
where 
\begin{equation*}
(t, p)\sim' (s,q) ~~ \mbox{if and only if} ~~ p=q ~~ \mbox{and} ~~ t^{-1}s \in T_p. 
\end{equation*}
Here, $T^{n+1}$ acts on $N(Q, \xi)$ induced by the multiplication on the first factor
of $T^{n+1} \times Q$.

\begin{proposition}\label{prop_two_definitions_are_equiv}
Let $(Q, \xi)$ be a hyper characteristic pair. Then the space $N(Q, \xi)$ in \eqref{eq_constr_DJ} is a quasi-contact toric manifold. 
\end{proposition}

\begin{proof}
Let $Z_Q$ be the moment angle manifold corresponding to $Q$, see \cite[Section 6.2]{BP-book}. Then $Z_Q$ is a smooth manifold and there is a smooth $T^m$-action on $Z_Q$ where $m$ is the number of facets of $Q$, see \cite[Corollary 6.2.5]{BP-book}. The space $N(Q, \xi)$ has a manifold structure which satisfy condition (1) and (2) in Definition \ref{def:axiom_top_cont}, see \cite[Proposition 2.3]{SaSo}. If rank of $\big<\xi_1, \ldots, \xi_{\ell}\big>$ is $n$, then $N(Q, \xi)$ is equvariantly homeomorphic to $M(Q, \lambda_{\xi}) \times S^1$ for some toric manifold $M(Q, \lambda_{\xi})$, see \cite[Proposition 2.6]{SaSo}.  If rank of $\big<\xi_1, \ldots, \xi_{\ell}\big>$ is $n+1$, then $N(Q, \xi)$ is equvariantly homeomorphic to $Z_Q/T_{\xi}$ for some $(m-n-1)$-dimensional subgruop $T_{\xi}$ of $T^m$, see \cite[Proposition 2.7]{SaSo}. Also, $M(Q, \lambda_{\xi})$ is equivariantly homeomorphic to $Z_Q/T_{\lambda_{\xi}}$ for some  some $(m-n)$-dimensional subgruop $T_{\lambda_{\xi}}$ of $T^m$. Therfore, $N(Q, \xi)$ has a smooth structure such that the $T^{n+1}$-action $N(Q, \xi)$ is smooth. 
\end{proof}

Note that the function $\eta$ defined in \eqref{eq_axiomatic_lambda} satisfies Definition  \ref{def_hyp_char_map}, see the explanation in \cite[Subsection 2.1]{SaSo}. Also the orientation of a quasi-contact toric manifold $N$ can be induced from an orientation of $Q$ and $T^{n+1}$.

\begin{proposition}\cite{SaSo}\label{cor_two_def_are_equiv}
Let $N$ be a quasi-contact toric manifold over the $n$-dimensional simple polytope $Q$, and $\eta$ a function as defined in \eqref{eq_axiomatic_lambda}. Then $N$ is equivariantly diffeomorphic to $N(Q, \eta)$. 
\end{proposition}
\begin{proof}
    This follows by similar arguments in the proofs of Lemma 1.4 and Proposition 1.8 in \cite{DJ} and Proposition \ref{prop_two_definitions_are_equiv}.
\end{proof}

\begin{example}[Generalized lens spaces]\label{ex_gen_lens_sp}
Let $\Delta^n$ be the $n$-dimensional simplex and $\xi$ a hyper characteristic function on it. We assume that the rank of $\big<\{\xi(E) \mid E\in \mathcal{F}(\Delta^n)\}\big> \subseteq \ZZ^{n+1}$ is $(n+1)$.
The article \cite{SS} shows that $N(\Delta^n,\xi)$ is equivariantly homoeomorphic (hence diffeomorphic) to the orbit space $S^{2n+1}/ G_ \xi$ for some free action of a finite group $G_{\xi}$. This space is called a \emph{generalized lens space} in \cite{SS}.  
In particular, if $\{\xi(E) \mid E\in \mathcal{F}(\Delta^n)\}$ forms a basis of $\ZZ^{n+1}$, then $N(\Delta^n, \xi)$ is homeomorphic to $S^{2n+1}$.

Consider an integer $p>1$ and $n$ integers $q_1,\dots,q_n$ such that $\gcd\{p,q_i\}=1$ for all $i=1,\dots,n$. Then $\ZZ_{p}$ acts freely on $S^{2n+1}$ by the following $$g(z_0,z_1,\dots,z_n)=(gz_0,g^{q_1}z_1,\dots,g^{q_n}z_n).$$ 
The lens space $L(p; q_1,\dots,q_n)$ is defined to be the orbit space $S^{2n+1}/\ZZ_p$. The paper \cite{SS} showed
that there is a hyper characteristic function $\xi$ on $\Delta^n$ such that $L(p; q_1,\dots,q_n)$ is equivariantly diffeomorphic to $N(\Delta^n, \xi)$.\qed 
\end{example}

\begin{example}[Good contact toric manifolds]\label{ex_gd_ct_toric_mfd}
Luo \cite[Chapter 2]{Luo-Thesis} discussed the construction of good contact toric manifolds which are compact connected contact toric manifolds studied in \cite[Section 2]{Lerman}. We briefly, recall the construction. Let $Q$ be an $n$-dimensional simple lattice polytope embedded in $\RR^{n+1}\setminus \{\mathbf{0}\}$. Consider the cone $C(Q)$ on $Q$ with apex $\mathbf{0}\in \mathbb{R}^{n+1}$ and  the set $\{\widetilde{E} \mid E\in \mathcal{F}(Q)\}$ of facets of $C(Q)$ where $\widetilde{E}\colonequals C(E)\setminus \{0\}$. Let $\xi(E)$ be the primitive outward normal vector on $\widetilde{E}$. This defines a function $\xi \colon \mathcal{F}(Q) \to \mathbb{Z}^{n+1}$. Since the facets of $C(Q)-\{0\}$ intersects transversely, the function $\xi$ satisfies Definition \ref{def_hyp_char_map}. Then the space $N(Q, \xi)$ is $T^{n+1}$-equivariantly homeomorphic to a good contact toric manifold whose moment cone is  $C(Q)$. Moreover, a good contact toric manifold can be obtained in this way. For details, we refer \cite{Lerman} and \cite{Luo-Thesis}. 
\qed
\end{example}

 \begin{lemma}\label{lem_out_pt}
 Let $Q$ be an $n$-dimensional simple polytope and $$\xi \colon  \mathcal{F}(Q) \to \ZZ^{n+1}$$ a hyper characteristic function. Then there exists ${\bf a}=(a_1,\dots,a_{n+1})\in \ZZ^{n+1}$ such that
 $\{{\bf a},\xi(E_{j_1}), \dots, \xi(E_{j_n})\}$  is a linearly independent subset in $\ZZ^{n+1}$ whenever $E_{j_1}\cap\cdots\cap E_{j_n}$ is a vertex of $Q$.
 \end{lemma}

  \begin{proof}
  Let $b$ be a vertex of $Q$. Then $b=E_{j_1}\cap\cdots\cap E_{j_n}$ for some unique facets $E_{j_1},\cdots, E_{j_n}$.
  Let $\ZZ_b$ be the submodule of $\ZZ^{n+1}$ generated by $\xi_{j_1}, \dots, \xi_{j_n}$. So the rank of $\ZZ_b$ is $n$ for any vertex $b\in V(Q)$. Since there are only finitely many vertices in $Q$ we have $\ZZ^{n+1}\setminus \bigcup_{b\in V(Q)}\ZZ_b$ is non-empty. If possible, let 
   \begin{equation}\label{eq_ten}
   \ZZ^{n+1}=\bigcup_{b\in V(Q)}\ZZ_b.
   \end{equation}
   Let $V_b:=\ZZ_b\otimes_{\ZZ}\RR$. Then $V_b$ is an $n$-dimensional linear subspace of $\RR^{n+1}$. From \eqref{eq_ten} it follows that
  $$\RR^{n+1}=\ZZ^{n+1}\otimes_{\ZZ} \RR =\bigcup_{b\in V(Q)}(\ZZ_b\otimes_{\ZZ} \RR) =\bigcup_{b\in V(Q)}V_b,$$ which is a contradiction as $V(Q)$ is a finite set.
  Let ${\bf a}\in \ZZ^{n+1}\setminus \bigcup_{b\in V(Q)}\ZZ_b$ be primitive. Then ${\bf a}$ is the desired vector of the lemma.
  \end{proof}

\begin{theorem}\label{thm_zoro_cob}
Let $N$ be a $(2n+1)$-dimensional quasi-contact toric manifold. Then there is a smooth oriented $T^{n+1}$-manifold with boundary $M$ such that $\partial{M}$ is equivariantly diffeomorphic to $N$.
\end{theorem}

\begin{proof}
There exist a hyper characteristic pair $(Q,\xi)$ such that $N$ is equivariantly diffeomorphic to $N(Q,\xi)$ by Proposition \ref{cor_two_def_are_equiv}. Let $\mathcal{F}(Q)\colonequals \{E_1, \dots, E_\ell\}$ and $I=[0,1]$. Then $Q\times I$ is an $(n+1)$-dimensional simple polytope and $\mathcal{F}(Q\times I)\colonequals \{E_1\times I, \dots, E_\ell\times I,Q\times\{0\},Q\times\{1\}\}$. Let ${ \bf a}\in \ZZ^{n+1}$ satisfies Lemma \ref{lem_out_pt}. Define a map $\lambda \colon  \mathcal{F}(Q\times I) \to \ZZ^{n+1}$ by
	\begin{align}\label{eq_new_lmd}
		\lambda(F)=\begin{cases} \xi(E_i) \quad &\text{ if } F=E_i\times I \text{ for }i =1,2,\dots,l\\
		{\bf a} \quad &\text{ if } F=Q\times \{0\}, Q\times \{1\}. \end{cases}
	\end{align}
 Then $\lambda$ is an $\mathcal{R}$-characteristic function on $P:=Q\times I$, and $M(P, \lambda)$ is a toric orbifold.  Theorem \ref{thm_res_sing} gives a resolution of singularities 
 \begin{equation}\label{eq_res_sing}
    M(P^{(d)},\lambda^{(d)})\to \dots \to M(P^{(1)},\lambda^{(1)})\to M(P,\lambda)  
 \end{equation}
 	for $M(P,\lambda)$, where $M(P^{(d)},\lambda^{(d)})$ is a  toric manifold and the arrows represent the associated blowups.
	
	Note that for any face $E$ of $Q$, the codimension of $E$ in $Q$ is same as the codimension of the face $E\times I$ in $P$. Moreover, if $$E=\bigcap_{s=1}^{k} E_{j_s}$$ for some unique facets $ E_{j_1}, E_{j_2},\dots, E_{j_k}$ of $Q$, then $E_{j_1}\times I,E_{j_2}\times I,\dots,E_{j_k}\times I$ are facets of $P$ and $$E\times I=\bigcap_{s=1}^{k} (E_{j_s}\times I).$$ 
	
	Now $\lambda(E_{j_s}\times I)=\xi(E_{j_s})$ and  $\{\xi(E_{j_1}),\xi(E_{j_2}),\dots,\xi(E_{j_k})\}$ is a direct summand of $\ZZ^{n+1}$. Then we have $$|G_{E\times I}(P,\lambda)|=1$$ for any face $E\times I$ of $P$ where $E$ is a face of $Q$. Therefore, if  $|G_{F}(P,\lambda)|\neq 1$ for a face $F$ of $P$, then either  $F\subseteq Q\times \{0\}$ or $F\subseteq Q\times \{1\}$. Thus we have $|G_{(E\times I)\cap P^{(j)}}(P^{(j)},\lambda^{(j)})|=1$ for every $j=1, \dots, d$. Therefore, the blowups in the resolution \eqref{eq_res_sing} need to take only on the faces arising from $Q\times\{0\}$ or $Q\times \{1\}\subset P$.  
Since $P$ is a simple polytope, one can consider the necessary blowups in the resolution \eqref{eq_res_sing} such that $Q\times [\frac{1}{2}-\delta,\frac{1}{2}] \subset P^{(d)}$, for some $\delta>0$. Let 
	\begin{equation*}
	    \widetilde{P} :=\iota^{-1}(Q\times [0,1/2])\subset P^{(d)},
	   	\end{equation*}
	where $\iota$ is the inclusion map $$\iota\colon P^{(d)} \to P.$$
	Then $\widetilde{P}$ is a simple polytope. Let
	$$\pi' \colon M(P^{(d)}, \lambda^{(d)} )=(P^{(d)}\times T^{n+1})/{\sim} \longrightarrow P^{(d)}$$
		be the quotient map. We prove that $(\pi')^{-1}(\widetilde{P})$ is a $T^{n+1}$-manifold with boundary, where the boundary is $N(Q,\xi )$.
	
	Let $\alpha$ be a point in $\widetilde{P}$ and $\alpha \notin \{(x,\frac{1}{2}): x\in Q\}$. Since $M(P^{(d)}, \lambda^{(d)})$ is a smooth manifold, then there exists a neighbourhood $U_{\alpha}$ of $\alpha$ in $\widetilde{P}$ such that $U_{\alpha}$ does not intersect $ \{(x,\frac{1}{2}): x\in Q\}$. So $(\pi')^{-1}(U_{\alpha})$ is an open subset of the smooth manifold $M(P^{(d)}, \lambda^{(d)})$.
	
	Now if $\alpha=(x,\frac{1}{2})\in \widetilde{P}$ for some $x\in Q$, then consider the tubular neighbourhood $Q\times (\frac{1}{2}-\delta,\frac{1}{2}] $ of $Q\times \{\frac{1}{2}\}$ in $\widetilde{P}$. Note that $\{(x,\frac{1}{2}): x\in Q\}=Q\times \{\frac{1}{2}\}$ which is nothing but $Q$. Then 
		$$(Q\times (\frac{1}{2}-\delta,\frac{1}{2}]\times T^{n+1})/{\sim}=  \Big{(}(Q\times \{\frac{1}{2}\} \times T^{n+1})/{\sim'}\Big{)}\times (\frac{1}{2}-\delta,\frac{1}{2}].$$
	
	 The characteristic function on $P^{(d)}$ induces a hyper characteristic function on the facets of  $Q\times \{\frac{1}{2}\}$ which is same as the hyper characteristic function on the facets of $Q$. This implies that $\Big{(}(Q \times\{\frac{1}{2}\}\times T^{n+1})/{\sim'}\Big{)}$ is equivariantly diffeomorphic to $N(Q,\xi)$. Thus the manifold with boundary  $(\pi')^{-1}(\widetilde{P})$ satisfies our claim. 
	Hence a quasi-contact toric manifold is equivariantly cobordant to zero.
\end{proof}

\begin{figure}
    \begin{center}
        \begin{tikzpicture}[scale=.6]
        \begin{scope}[ yshift=80, xshift=-100]
       \draw (0,0)--(1,-2)--(3,-2)--(4,0)--(2,2)--(0,0);
       \node at (4.8,-1) {$(0,0,1)$};
       \node at (4.4,1) {$(1,1,0)$};
       \node at (-0.4,1) {$(1,1,1)$};
       \node at (-.8,-1) {$(0,1,1)$};
        \node at (2,-2.5) {$(1,0,0)$};
         \node at (2,-3.2) {(A)}; 
        \end{scope}
        
        \begin{scope}[xshift=-180, yshift=-140]
        \draw [thick, blue ](0,0)--(2,0);
        \draw (2,0)--(3,1)--(1,2)--(-1,1)--(0,0);
         \draw[fill=yellow, yellow, opacity=0.3](0,-1.3)--(2,-1.3)--(3,-0.3)--(1,0.7)--(-1,-.3)--(0,-1.3);
    \draw[dotted, thick](0,-1.3)--(2,-1.3)--(3,-0.3)--(1,0.7)--(-1,-.3)--(0,-1.3);
        \draw (-1,1)--(-1,-3)--(0,-4);
        \draw (2,-4)--(3,-3)--(3,1);
        \draw [thick, blue](0,-4)--(2,-4);
        \draw (0,-4)--(0,0);
        \draw (2,-4)--(2,0);
        \draw[dashed] (-1,-3)--(1,-2)--(3,-3);
        \draw [dashed] (1,2)--(1,-2);
        
        \draw[->] (1.5,1)--(1.7,3);
        \node[above] at (1.7,3) {$(1,2,0)$};
        
        \draw[dotted, thick] (.5,1)--(0,1.5);
        \draw[->] (0,1.5)--(-1.5,3);
        \node[left] at (-1.5,3) {$(1,1,1)$};
        
        \draw[->] (-.5,-.5)--(-2,-.5);
        \node[left] at (-2,-.5) {$(0,1,1)$};
       \node at (1,-4.6) {(B)};

        \draw[->] (.5,-2) to [out=220, in=330] (-2,-2);
        \node[left] at (-2,-2) {$(1,0,0)$};
        
        \draw[dotted, thick] (.6,-3) to [out=260, in=340] (-.4,-3.57);
        \draw[->] (-.3,-3.6)--(-2,-3.6);
        \node[left] at (-2,-3.6) {$(1,2,0)$};
        
        \draw[dotted, thick] (2.5,-.5)--(3,0);
        \draw[->] (3,0)--(4,1);
        \node[right] at (3.8,1.2) {$(1,1,0)$};
        
        \draw[->] (2.5,-1.5)--(4,-1.5);
        \node[right] at (4,-1.5) {$(0,0,1)$};
        
        \draw[->] (1.25,0)--(3,2.25);
        \node[right] at (3,2.5) {$F_1$};
        
        \draw[->] (1,-4) to [out=330, in=200] (4, -4);
        \node[right] at (4,-4) {$F_0$};
        \end{scope}
        
        \begin{scope}[xshift=120, yshift=-140]
        \draw (0,0)--(2,0)--(3,1)--(1,2)--(-1,1)--(0,0);
      \draw[dotted, thick](0,-1.3)--(2,-1.3)--(3,-0.3)--(1,0.7)--(-1,-.3)--(0,-1.3);
      \draw[fill=yellow, yellow, opacity=0.3](0,-1.3)--(2,-1.3)--(3,-0.3)--(1,0.7)--(-1,-.3)--(0,-1.3);
      
      \draw[fill=blue, blue, opacity=0.3](0,-3.3)--(2,-3.3)--(2.2,-4)--(-.2,-4)--(0,-3.3);
      
        \draw (-1,1)--(-1,-3)--(-.2,-4)--(2.2,-4)--(3,-3)--(3,1);
        \draw (2,-3.3)--(0,-3.3)--(0,0);
        \draw (2.2,-4)--(2,-3.3)--(2,0);
        \draw (0,-3.3)--(-.2,-4);
        \draw[dashed] (-1,-3)--(1,-2)--(3,-3);
        \draw [dashed] (1,2)--(1,-2);
        
        \draw[->] (1.5,1)--(1.7,3);
        \node[above] at (1.7,3) {$(1,2,0)$};
        
        \draw[dotted, thick] (.5,1)--(0,1.5);
        \draw[->] (0,1.5)--(-1.5,3);
        \node[left] at (-1.5,3) {$(1,1,1)$};
        
        \draw[->] (-.5,-.5)--(-2,-.5);
        \node[left] at (-2,-.5) {$(0,1,1)$};
        
        \draw[->] (.5,-2) to [out=220, in=330] (-2,-2);
        \node[left] at (-2,-2) {$(1,0,0)$};
        
        \draw[dotted, thick] (.6,-3) to [out=260, in=340] (-.4,-3.57);
        \draw[->] (-.3,-3.6)--(-2,-3.6);
        \node[left] at (-2,-3.6) {$(1,2,0)$};
        
        \draw[->] (1.25,0)--(3,2.25);
        \node[right] at (3,2.5) {$F_1$};
        
        \draw[dotted, thick] (2.5,-.5)--(3,0);
        \draw[->] (3,0)--(4,1);
        \node[right] at (3.8,1.2) {$(1,1,0)$};
        
        \draw[->] (2.5,-1.5)--(4,-1.5);
        \node[right] at (4,-1.5) {$(0,0,1)$};
        
        \draw[->] (1,-3.7) to [out=330, in=200] (4, -4);
        \node[right] at (4,-4) {$(1,1,0)$};
          \node at (1,-4.5) {(C)}; 
        \end{scope}

        \end{tikzpicture}
    \caption{Procedure of constructing a manifold whose boundary is a given quasi-contact toric manifold.}
    \label{Fig cobor_q_mfd}
     \end{center}
\end{figure}
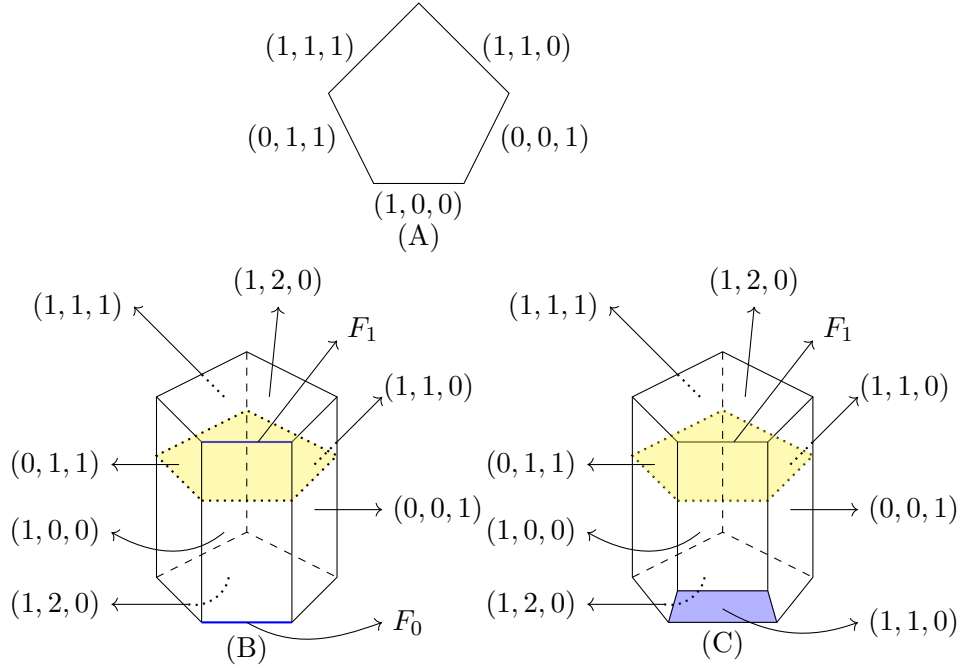  
\begin{example}
  Consider the pentagon $Q$ and the hyper characteristic function 
   $$\xi\colon \mathcal{F}(Q)\to \ZZ^3$$
   as in Figure \ref{Fig cobor_q_mfd}(A). Then we have the quasi-contact toric manifold $N(Q,\xi)$. There exists a point ${\bf a}\in \ZZ^3$ which satisfies Lemma \ref{lem_out_pt} for this pair $(Q, \xi)$. In this case one can take ${\bf a}=(1,2,0)$. Then using \eqref{eq_new_lmd} we can define an $\mathcal{R}$-characteristic function $\lambda$ on the pentagonal prism $P:=Q\times I$, see Figure \ref{Fig cobor_q_mfd}(B). Thus, we have the toric orbifold $M(P, \lambda )$. 
  
  In the simple polytope $P$, the only faces $F$ with $|G_{F}|\neq 1$ are two edges $F_0$ and $F_1$ which are coloured with blue in Figure \ref{Fig cobor_q_mfd}(B) and the $4$ vertices of these edges. Using the techniques of the proof of Theorem \ref{thm_res_sing}, first we blowup $P$ along the face $F_0$ and get the simple polytope $P'$ as in \ref{Fig cobor_q_mfd}(C). We denote the corresponding blowup of toric orbifold $M(P, \lambda )$ by $M(P',\lambda')$. Note that after this blowup the only face $F$ in the polytope $P'$ with $|G_{F}|\neq 1$ is $F_1$ and the two vertices of $F_1$, see figure \ref{Fig cobor_q_mfd}(C).
  
  Note that $Q\times \{\frac{1}{2}\} \subseteq P' \subseteq P $ which is denoted by the dotted pentagon filled with yellow in Figure \ref{Fig cobor_q_mfd}(B) and \ref{Fig cobor_q_mfd}(C). Also $Q\times \{\frac{1}{2}\}$ can be identified with $Q$. Let $\widetilde{P}$ denotes the lower portion of this pentagon $Q\times \{\frac{1}{2}\}$ in $P'$, that is $\widetilde{P} = (Q \times [0, \frac{1}{2}]) \cap P'$. Then $N(Q,\xi)$ is equivariantly the boundary of the oriented smooth manifold $(\pi')^{-1}(\widetilde{P})$ with boundary  for this example. \qed
  \end{example}

\begin{corollary}
Any good contact toric manifold is equivariantly the boundary of an oriented smooth manifold.
\end{corollary}
\begin{proof}
This follows from Example \ref{ex_gd_ct_toric_mfd} and Theorem \ref{thm_zoro_cob}.
\end{proof}

\begin{corollary}\label{cor_lens_sp}
Any generalized lens space is equivariantly the boundary of an oriented smooth manifold.
\end{corollary}
\begin{proof}
This follows from Example \ref{ex_gen_lens_sp} and Theorem \ref{thm_zoro_cob}.
\end{proof}

We note that the conclusion of Corollary \ref{cor_lens_sp} can be obtained from \cite{Han}. However, our proof is more geometric and explicit. We also note that the article \cite{SS} gave partial answer to Corollary \ref{cor_lens_sp} under some number theoretic sufficient conditions.

\begin{remark}
Using Theorem \ref{thm_zoro_cob} and \cite[Theorem 4.9]{MS}, we  get that all Stiefel-Whitney numbers of quasi-contact toric manifolds are zero.
\end{remark}

\subsection*{Acknowledgments}
The authors would like to thank Dong Youp Suh and Jongbaek Song for helpful discussion. The first author thanks `IIT Madras' for PhD fellowship. The second author thanks `International office IIT Madras' and Science and Engineering Research Board India for research grants. The third author thanks to IIT Madras for PDEF fellowship and IMSc for PDF fellowship.

\bibliographystyle{abbrv}

\end{document}